\documentclass[10pt]{amsart}
\usepackage{geometry}

\usepackage{amssymb}
\usepackage{amsmath}
\usepackage{amscd}
\usepackage{amsthm}
\usepackage[dvips]{graphicx}

\usepackage[all]{xy}
\usepackage{extarrows}
\usepackage{rotating}

\usepackage{galois}
\usepackage{mathrsfs}

\renewenvironment{proof}{{ \textbf{Proof}.}}{\qed}
\newtheorem{thm}{Theorem}[section]
\newtheorem{lem}[thm]{Lemma}
\newtheorem{Def}[thm]{Definition}
\newtheorem{cor}[thm]{Corollary}
\newtheorem{prop}[thm]{Proposition}
\newtheorem{Ex1}[thm]{Example}
\newtheorem{Rem1}[thm]{Remark}

 \normalbaselines
   \newcommand{\dR}{\mathrm{dR}}

\usepackage[colorlinks=blue, linkcolor=green]{hyperref}
\usepackage{hyperref}

\title{Hodge-type 
decomposition for de Rham cohomology of $ p 
	$-adically 
	uniformized varieties}
\author{Yufan Luo}
\date{\today}

\address{Yufan Luo
\newline School of Mathematical Sciences,
\newline East China Normal University
\newline 200241, Shanghai,
\newline P. R. China}
\email{yufanluo@hotmail.com}

\begin{document}	
\maketitle

 \begin{abstract}
 	We prove a Hodge-type decomposition for the de-Rham 
 	cohomology 
 	of $ p 
 	$-adically 
 	uniformized varieties by the product of Drinfeld's symmetric spaces. It is 
 	based on work of Schneider, Stuhler, Iovita and Spiess on the cohomology of 
 	Drinfeld's symmetric space.
 \end{abstract}

 \section{Introduction}
 Let $ K $ be a $ p $-adic field, i.e. a finite extension of $ 
 \mathbb{Q}_{p} $ 
 and let $ d $ be a 
 natural number. The $ d $-dimension \textit{Drinfeld's symmetric space} 
 over $ 
 K $, 
 denoted by $ \mathfrak{X}^{d}_{K} $, is the complement in $ 
 \mathbb{P}_{K}^{d} 
 $ of 
 all $ K $-rational hyperplanes. It is a rigid analytic Stein space. We have 
 a 
 natural action of $ G:=PGL_{d+1}(K) $ on $ \mathfrak{X}_{K}^{d} $. 
 
 Let $ \Gamma $ be a discrete torsion-free, cocompact subgroup of $ G $. We 
 have 
 a natural action of $ \Gamma $ on $ \mathfrak{X}_{K}^{d} $. By a theorem of 
 Mustafin \cite{3}, 
 the 
 quotient $ 
 X_{\Gamma}:=\Gamma \backslash 
 \mathfrak{X}_{K}^{d} $ is the rigid analytic space associated to a unique 
 smooth 
 projective variety defined over $ K $. For every $ 0\leq n\leq d $, the de 
 Rham 
 cohomology $ H_{\dR}^{n}(X_{\Gamma}) $ of $ X_{\Gamma} $ 
 admits the usual Hodge filtration $ 
 F_{\dR}^{\bullet}H_{\dR}^{n}(X_{\Gamma}) $ 
 and covering filtration $ F_{\Gamma}^{\bullet}H_{\dR}^{n}(X_{\Gamma}) $ 
 induced 
 by the covering speactral sequence. It is known that the filtrations $  
 F_{\dR}^{\bullet}H_{\dR}^{n}(X_{\Gamma})  $ and $  
 F_{\Gamma}^{\bullet}H_{\dR}^{n}(X_{\Gamma}) $ are opposite to each other, 
 i.e. 
 \[  H_{\dR}^{n}(X_{\Gamma})=F_{\Gamma}^{p}\oplus F_{\dR}^{n+1-p} \]
 for any $ 
 0\leq p\leq n+1 $. As 
 consequence one gets a Hodge-type decomposition 
 \[  H^{n}_{\dR}(X_{\Gamma})=\bigoplus_{p+q=n}F_{\dR}^{p}H_{\dR}^{n}\cap 
 F_{\Gamma}^{q}H^{n}_{\dR}. \]
 This Hodge-type decomposition has been conjectured by Schneider \cite{2} 
 and 
 proved by 
 Iovita and Spiess \cite{6}. Other proofs were given later in \cite{5} and 
 \cite{7}. 
 
 In this 
 paper, we will generalize the Hodge-type decomposition: we will consider a 
 variety $ X $ which is uniformized by the product of the Drinfeld's symmetric 
 spaces. Interesting examples of such varieties are the unitary 
 Shimura 
 varietes studied by Rapoport and Zink (See \cite{RZ}, Chapter 6).
 
 Our setting is as follows. Let $ K,K' $ be $ p $-adic fields and $ K'' $ a 
 $ p 
 $-adic field containing $ K $ and $ K' $. Let $ 
 \mathcal{O}_{K},\mathcal{O}_{K'},\mathcal{O}_{K''} $ denote their valuation 
 ring respectively. Fix integers 
 $ d,d'\geq 1 $ and put $ G=PGL_{d+1}(K),~G'=PGL_{d'+1}(K') $ respectively. 
 Let $ 
 \mathfrak{X}_{K}= \mathfrak{X}^{d}_{K} $ (resp. $ 
 \mathfrak{X}_{K'}=\mathfrak{X}^{d'}_{K'} $) 
 denote 
 the $ d $-dimensional (resp. $ d' $-dimensional) Drinfeld's symmetric 
 space. 
 To simplify the notation, we write $ 
 \mathfrak{X}_{K}\times_{K''}\mathfrak{X}_{K'}:=(\mathfrak{X}_{K}\otimes_{K}K'')\times_{K''}
 (\mathfrak{X}_{K'}\otimes_{K'}K'') $. 
 
 The main theorem of this paper is the following:
 \begin{thm}\label{main}\label{thm1.1}
 	Let $ \Gamma$ be a discrete torsion-free cocompact 
 	subgroup of $ G\times G' $, the quotient $X_{\Gamma}:=\Gamma \backslash 
 	(	\mathfrak{X}_{K}\times_{K''}\mathfrak{X}_{K'})$ becomes a smooth 
 	rigid 
 	variety 
 	over $ K'' $. For $ 0\leq n\leq d+d' $, there exists a covering 
 	spectral sequence for $ H_{\dR}^{n}(X_{\Gamma}) $ and it degenerates at 
 	$ 
 	E_{2} $. Moreover, the assosicated 
 	covering filtration $ F_{\Gamma}^{\bullet} $ is 
 	opposite to the usual Hodge filtration $ F_{\dR}^{\bullet} $ 
 	on $ 
 	H^{n}_{\dR}(X_{\Gamma}) $, i.e. $ 
 	H_{\dR}^{n}(X_{\Gamma})=F_{\Gamma}^{i}\oplus F_{\dR}^{n+1-i} $ for any 
 	$ 
 	0\leq i\leq n+1 $. Hence we have a Hodge-type 
 	decomposition for $ H_{\dR}^{n}(X_{\Gamma}) $
 	\[  H^{n}_{\dR}(X_{\Gamma})=\bigoplus_{p+q=n}F_{\dR}^{p}\cap 
 	F_{\Gamma}^{q}. \]
 \end{thm}
 
 Our strategy is the same as that in \cite{1},\cite{2} and \cite{6}. First 
 we 
 compute the 
 de Rham cohomology of $ 
 \mathfrak{X}_{K}\times_{K''}\mathfrak{X}_{K'} $ by Kunneth formula and the 
 results of Schneider and Stuhler \cite{1}. Second we prove 
 the existence of covering spectral sequence for de Rham cohomology of the
 quotient 
 $ \Gamma \backslash X $, where $ X $ is any 
 separated smooth rigid space over $ K $ and $ \Gamma $ is a 
 group of automorphisms acting discontinuously, freely on $ X $. In 
 particular, the 
 existence 
 of 
 covering spectral sequence 
 for de Rham cohomology of quotient space $ \Gamma \backslash 
 (\mathfrak{X}_{K}\times_{K''}\mathfrak{X}_{K'})  $ is established, where $ 
 \Gamma $ is a 
 discrete torsion-free subgroup of $ G\times G' $. The $ E_{2} $-terms of 
 this 
 covering spectral sequence 
 can 
 be 
 transformed into Ext-groups in the category of smooth $ G\times G' 
 $-representation and we compute these Ext-groups as same as in \cite{1}. 
 Then 
 we can also
 compute the cohomology of quotient space $ X_{\Gamma} $.
 Using  
 this result, finally we can prove a Hodge-type decomposition for the 
 de Rham 
 cohomology 
 of $ \Gamma \backslash 
 (\mathfrak{X}_{K}\times_{K''}\mathfrak{X}_{K'})  $ as same as in \cite{6}.

 Although Theorem \ref{main} is stated for the product of two Drinfeld's 
 symmetric 
 spaces, our argument also works for the product of more than two Drinfeld's 
 symmetric spaces. 
 
 \section{The de Rham cohomology of $ 
 	\mathfrak{X}_{K}\times_{K''}\mathfrak{X}_{K'} $ }
 
 In this section we want to compute the de Rham cohomology of $ 
 \mathfrak{X}_{K}\times_{K''}\mathfrak{X}_{K'} $. 
 
 First we recall a result of Schneider and Stuhler, which is crucial for 
 our work. In \cite{1}, they studied the cohomology of Drinfeld's 
 symmetric space for any cohomology theory satisfying certain natural axioms. 
 Examples are de Rham cohomology and $ \ell $-adic ($ \ell \neq p $) cohomology.
 We need some more notations. For any subset $ 
 I\subset \Delta:=\{1,2,\cdots,d\} $ let $ P_{I}\subset G=PGL_{d+1}(K) $ be 
 that 
 parabolic subgroup which stabilizes the flag
 \[ \left( \sum_{i=1}^{i_{0}}Ke_{i}\subsetneq 
 \sum_{i=1}^{i_{1}}Ke_{i}\subsetneq 
 \cdots \subseteq \sum_{i=1}^{i_{r}}Ke_{i}\right)  \]
 in $ K^{d+1} $ where $ \{i_{0}<i_{1}<\cdots <i_{r}\}=\Delta \backslash I $ and 
 $ e_{1},\cdots,e_{d+1} $ is the standard basis. Furthermore if $ A $ is any 
 abelian group let $ C^{\infty}(G/P_{I},A) $ be the group of all $ A $-valued 
 locally constant functions on the space $ 
 G/P_{I} $. We put
 \[ V_{I}(A):=C^{\infty}(G/P_{I},A)/\sum_{i\in \Delta \backslash 
 	I}C^{\infty}(G/P_{I\cup \{i\}},A)=V_{I}(\mathbb{Z})\otimes_{\mathbb{Z}} A \]
 which is an irreducible smooth representation of $ G $ (see \cite{1}).  For $ 
 0\leq s\leq d $ 
 we use the 
 simpler natation
 \[ V_{s}(A):=V_{\{1,2,\cdots,d-s\}}(A) .\]
For $ s>d $, we set $ V_{s}(A)=0 $. Similarly, for any subset $ J\subset 
\Delta':=\{1,2,\cdots,d'\} $ we consider 
 an irreducible smooth representation $ V'_{J}(A) $ of $ G' $.
 
 \begin{lem}(\cite{1}, Theorem 1 and Lemma 1 in Section 4)\label{lem2.1}
 	For $ 0\leq s\leq d $ there is a natural $ G $-equivariant isomorphism
 	\[ H_{\dR}^{s}(\mathfrak{X}_{K})=\text{Hom}_{K}(V_{s}(K),K) \]
 	and for $ s>d $ we have $ H_{\dR}^{s}(\mathfrak{X}_{K})=0 $. Similarly, for 
 	$ 
 	0\leq s'\leq d' $ there is a natural $ G' $-equivariant isomorphism
 	\[ H_{\dR}^{s'}(\mathfrak{X}_{K'})=\text{Hom}_{K'}(V'_{s'}(K'),K') \]
 	and for $ s'>d' $ we have $ H_{\dR}^{s'}(\mathfrak{X}_{K'})=0 $.
 \end{lem}
 
 \begin{prop}\label{prop2.2}
 	The Kunneth formula holds for the de Rham cohomology of the product of 
 	Drinfeld's 
 	symmetric spaces $ \mathfrak{X}_{K}\times_{K''}\mathfrak{X}_{K'} $. That is,
 	for $ 0\leq n\leq d+d'  $ we have a natural $ G\times G' $-isomorphism
 	\[ H^{n}_{\dR}(\mathfrak{X}_{K}\times_{K''}\mathfrak{X}_{K'})\simeq 
 	\bigoplus_{i+j=n}H^{i}_{\dR}(\mathfrak{X}_{K}\otimes_{K} 
 	K'')\otimes_{K''}H^{j}_{\dR}(\mathfrak{X}_{K'}\otimes_{K'} K'') \]
 	and for $ n>d+d' $ we have $ 
 	H^{n}_{\dR}(\mathfrak{X}_{K}\times_{K''}\mathfrak{X}_{K'})=0 $.
 \end{prop}
 \begin{proof}
 	Recall that $\mathfrak{X}_{K} $ (resp. $\mathfrak{X}_{K'} $) is Stein-space 
 	(\cite{1}, Proposition 4 in Section 1) 
 	and 
 	Theorem B (\cite{14})
 	says 
 	that Stein spaces have trivial 
 	coherent sheaf cohomology, so its de Rham 
 	cohomology is the cohomology of the complex of global rigid analytic froms 
 	$ \Omega^{\bullet}(\mathfrak{X}_{K}) $ (resp. $ 
 	\Omega^{\bullet}(\mathfrak{X}_{K'}) $). That is, 
 	\[ 
 	H_{\dR}^{n}(\mathfrak{X}_{K}\otimes_{K}K'')=h^{n}(\Omega^{\bullet}
 	(\mathfrak{X}_{K}\otimes_{K}K''))\]
 	resp. 
 	\[ 
 	H_{\dR}^{n}(\mathfrak{X}_{K'}\otimes_{K'}K'')=h^{n}(\Omega^{\bullet}
 	(\mathfrak{X}_{K'}\otimes_{K'}K''))\]
 	Let $ 
 	p_{1}:\mathfrak{X}_{K}\times_{K''}\mathfrak{X}_{K'}\to\mathfrak{X}_{K}
 	\otimes_{K}
 	K'',p_{2}:\mathfrak{X}_{K}\times_{K''}\mathfrak{X}_{K'}\to 
 	\mathfrak{X}_{K'}\otimes_{K'} K'' $ be the 
 	canonical projective morphisms, then we have a canonical isomorphism 
 	(\cite{15}, Proposition 1.4)
 	\[ 
 	p_{1}^{\ast}\Omega^{\bullet}_{(\mathfrak{X}_{K}\otimes_{K}K'')/K''}\otimes_{K''}
 	p_{2}^{\ast}\Omega^{\bullet}_{(\mathfrak{X}_{K'}\otimes_{K'}K'')/K''}\simeq
 	\Omega^{\bullet}_{(\mathfrak{X}_{K}\times_{K''}\mathfrak{X}_{K'})/K''}
 	\]
 	Note that we have an isomorphism of $ K''[G\times G'] $-modules
 	\begin{align*}
 	&\Gamma\left(\mathfrak{X}_{K}\times_{K''}\mathfrak{X}_{K'}, 
 	p_{1}^{\ast}\Omega^{\bullet}_{(\mathfrak{X}_{K}\otimes_{K}K'')/K''}\otimes_{K''}
 	p_{2}^{\ast}\Omega^{\bullet}_{(\mathfrak{X}_{K'}\otimes_{K'}K'')/K''}\right)
 	\\
 	\simeq&  \Gamma\left(\mathfrak{X}_{K}\otimes_{K}K'', 
 	\Omega^{\bullet}_{(\mathfrak{X}_{K}\otimes_{K}K'')/K''}\right) \otimes_{K''}
 	\Gamma \left(\mathfrak{X}_{K'}\otimes_{K'}K'',\Omega^{\bullet}
 	_{(\mathfrak{X}_{K'}\otimes_{K'}K'')/K''}\right)
 	\end{align*}

 	then by the Kunneth formula for complexes (\cite{12}, Theorem 3.6.3) we get 
 	the first assertion. The second assertion follows from the Lemma 
 	\ref{lem2.1}.
 \end{proof}
 
 \begin{cor}\label{cor2.3}
 	For $ 0\leq n\leq d+d'  $, there is a $ G\times G' $-isomorphism
 	\[  H^{n}_{\dR}(\mathfrak{X}_{K}\times_{K''}\mathfrak{X}_{K'})\simeq 
 	\bigoplus_{i+j=n}\text{Hom}_{K''}(V_{i}(K''),K'')\otimes_{K''}\text{Hom}_{K''}
 	(V'_{j}(K''),K'') .\]
 \end{cor}
 \begin{proof}
 	It follows from Proposition \ref{prop2.2} and Lemma 
 	\ref{lem2.1}.	
 \end{proof}

 \section{The quotient varieties $ \Gamma \backslash 
 	(\mathfrak{X}_{K}\times_{K''}\mathfrak{X}_{K'}) $ and Covering Spectral 
 	Sequence}
 
 Let $ X $ be a separated smooth rigid space over $ K $ and $ \Gamma $ be a 
 group of automorphisms acting discontinuously, freely on $ X $. We will prove 
 the existence of covering spectral sequence for de Rham cohomology of quotient 
 $ \Gamma \backslash X $.  Recall that the existence of covering spectral 
 sequence for $ H_{\dR}^{\ast}(\Gamma \backslash \mathfrak{X}_{K}) $ was 
 established in (\cite{1}, Proposition 2 in 
 Section 5), where $ \Gamma $ is a discrete torsion-free cocompact subgroup of 
 $ 
 G $. In loc. cit, they only considered a speccial case, but their 
 argument is valued for our general case.
 
 We recall that the action of a group $ \Gamma $ on a rigid space $ X $ over $ 
 K 
 $ 
 is said to 
 \textit{discontinuous} if $ X $ has an admissible affinoid covering $ 
 \{X_{i}\} 
 $ such that the set $ \{\gamma\in \Gamma~;~\gamma X_{i}\cap X_{i}\neq 
 \emptyset\} $ is finite for each $ i $. Furthermore, the above action is said 
 to be \textit{free} if $ \Gamma_{x}=\{1\} $ for all $ x\in X $.

 \begin{Def}(\cite{4}, Section 6.4)
 	Let a group $ \Gamma $ act on a rigid space $ X $ over $ K $. The 
 	\textit{quotient} $ 
 	p:X\to \Gamma \backslash X $ is defined by the following:
 	\begin{enumerate}
 		\item $ \Gamma \backslash X $ is a rigid space and $ p $ is a $ \Gamma 
 		$-invariant morphism of rigid spaces, i.e. $ p\circ \gamma=p $ for all 
 		$ \gamma \in \Gamma $.
 		\item Let $ U\subset X $ be an admissible, $ \Gamma $-invariant set, 
 		and let $ f:U\to Y $ be a morphism which is $ \Gamma $-invariant, i.e. 
 		$ f\circ \gamma =f $ for all $ \gamma \in \Gamma $. Then $ p(U)\subset 
 		\Gamma \backslash X $ is adimissible and the morphism factors uniquely 
 		over $ p(U) $, i.e.$ \xymatrix{f=g\circ p:U\ar[r]& p(U)\ar[r]^{g}& Y} $.
 	\end{enumerate}
 	We denote the quotient $ \Gamma \backslash X $ by $ X_{\Gamma} $.
 \end{Def}
 
 \begin{lem}(\cite{11}, Proposition 3 in Section 6.3.3)\label{lem3.2}
 	If $ G $ is a finite group of automorphisms of the $ K $-affinoid algebra $ 
 	A 
 	$, then $ A^{G} $ is also a $ K $-affinoid algebra and $ A $ is finite 
 	over $ A^{G} $. Furthermore, assume that the action of $ G $ is free, then 
 	the 
 	canonical morphism $ p:\text{Sp}(A)\to \text{Sp}(A^{G}) $ is an \'etale 
 	covering.
 \end{lem}
 
 \begin{thm}\label{thm3.3}
 	Let $ X $ be a separated rigid space over $ K $. Let $ \Gamma $ acts 
 	discontinuously 
 	and freely on $ X $. Then
 	\begin{enumerate}
 		\item The quotient $ X_{\Gamma} $ exists and it is seperated.
 		\item The canonical projective morphism $ p:X\to X_{\Gamma} $ 
 		is an \'etale covering. In particular, if $ X $ is a smooth rigid 
 		space, 
 		so does $ X_{\Gamma}$.
 		\item Moreover, we have a canonical isomorphism
 		\begin{align*}
 		\xymatrix@R=2ex{
 			\Gamma \times X=	\coprod_{g\in \Gamma}X\ar@{=>}[r]
 			& X\times_{X_{\Gamma}}X\\
 			(g,z)\ar@{|->}[r] & (gz,z).	
 		}
 		\end{align*}
 	\end{enumerate}  
 \end{thm}
 \begin{proof}
 	We briefly sketch the construction of the quotient. First of all 
 	$X_{\Gamma} $ is the set theoretical quotient and $ p:X\to X_{\Gamma} $ is 
 	the canonical surjective map. A subset $ U $ 
 	of $X_{\Gamma} $ is called admissible if $ p^{-1}(U) $ is 
 	admissible. A collection of admissible sets $ \{U_{i}\} $, with union 
 	the adimissible $ U $, is an admissible covering if $ \{p^{-1}(U_{i})\} 
 	$ is an admissible covering of $ p^{-1}(U) $. The structure sheaf $ 
 	\mathcal{O} $ on $ X_{\Gamma}$ is defined by 
 	\[\mathcal{O}(U):=\mathcal{O}_{X}(p^{-1}(U))^{\Gamma}.  \]
 	Obviously, $ p:X\to X_{\Gamma} $ then is a morphism of ringed 
 	spaces over $ K $.
 	
 	Since $ \Gamma $ acts discontinuously on $ X $, there exists an 
 	admissible affinoid covering  $ \{X_{i}\}_{i\in I} $ such that for every $ 
 	i\in 
 	I $, $ \{\gamma\in \Gamma~|~\gamma X_{i}\cap X_{i}\neq \emptyset\} $ is 
 	a finite set. It follows from the assumptions 
 	that $ \{X_{i}\}_{i\in I} $ satisfies the following two conditions:
 	\begin{itemize}
 		\item For all $ i $, $ \Gamma_{i}:=\{\gamma \in \Gamma~|~\gamma 
 		X_{i}=X_{i}\} $ is a finite group.
 		\item For any $ i,j\in I $, $ X_{j}\cap \gamma(X_{i})=\emptyset $ 
 		for all but finitely many $ \gamma\in \Gamma $.
 	\end{itemize}
 	Put $ Y_{i}:=p(X_{i}) $. The set $ Y_{i} $ is admissible since $ 
 	p^{-1}(Y_{i})=\cup_{\gamma \in \Gamma}\gamma X_{i} $ is admissible in $ X 
 	$: 
 	for any $ j\in I $, one has
 	\[ X_{j}\bigcap \left( \bigcup_{\gamma \in \Gamma}\gamma X_{i} \right) 
 	=\bigcup_{k=1}^{r}\left( \gamma_{k}X_{i}\cap X_{j}\right)  \]
 	is admissible open subset.
 	Moreover, we see that $ p $ induces an isomorphism of $ G $-ringed space
 	\[  \Gamma_{i}\backslash X_{i}\simeq Y_{i}. \]
 	Then by Lemma \ref{lem3.2} we see that $ \{Y_{i}\}_{i\in I} $ form an 
 	admissible covering of $ X_{\Gamma}$ by affinoid varities over $ K $; the 
 	intersection $ Y_{i}\cap Y_{j} $ are affinoid. Therefore $ X_{\Gamma}$ is a 
 	separated analytic variety over $ K $ and the canonical projective 
 	morphism $ p:X\to X_{\Gamma} $ is an \'etale covering by Lemma 
 	\ref{lem3.2} again. Thus we get the first and second assertion.
 	
 	The proof of the last assertion is the same as \cite{1} 
 	Theorem 2 in Section 5. (compare	\cite{8} Proposition 2 on p.70 ).
 	
 \end{proof}

 Let $ \mathcal{V} $ be the category of smooth separated rigid analytic 
 varieties over $ K $ equipped with a fixed Grothendieck topololgy which we 
 assume to be finer than the analytic topology. (See \cite{1})
 \begin{lem}
 	Let $ X $ be a separated smooth rigid space and group $ \Gamma $ acts 
 	discontinuously, freely on $ X $, then
 	\begin{enumerate}
 		\item For any sheaf $ \mathcal{F} $ on $ \mathcal{V} $, we have
 		\[ \mathcal{F}(X_{\Gamma})=\mathcal{F}(X)^{\Gamma} .\]
 		\item For any injective sheaf $ \mathcal{F} $ on $ \mathcal{V} $ the $ 
 		\Gamma $-module $ \mathcal{F}(X) $ is injective.
 	\end{enumerate}	
 \end{lem}
 \begin{proof}
 	The proof is the same as in \cite{1}, Lemma 1 in Section 5.
 \end{proof}

 \begin{prop}\label{prop3.5}
 	Let $ X $ be a separated smooth rigid space and group $ \Gamma $ acts 
 	discontinuously, freely on $ X $, then we have the \textit{covering 
 		spectral sequence}
 	\[ E_{2}^{r,s}=H^{r}(\Gamma,H^{s}_{\dR}(X))\Longrightarrow 
 	H_{\dR}^{r+s}(X_{\Gamma}). \]
 \end{prop}
 \begin{proof}
 	It follows from the above lemmas and Grothendieck spectral sequence. (See 
 	\cite{1}, Proposition 2 in Section 5)
 \end{proof}
 
 In particular, we have the following corollary.
 \begin{cor}
 	Let $ \Gamma $ be a 
 	discrete torsion-free subgroup of $ G\times G' $. Then the quotient 
 	$X_{\Gamma}:=\Gamma \backslash 
 	(\mathfrak{X}_{K}\times_{K''}\mathfrak{X}_{K'})$ becomes a smooth rigid 
 	variety 
 	over $ K'' $ and  there exists a covering 
 	spectral sequence for $ H_{dR}^{n}(X_{\Gamma}) $
 	\begin{align}
 	E_{2}^{r,s}=H^{r}(\Gamma,H_{\dR}^{s}(\mathfrak{X}_{K}\times 
 	_{K''}\mathfrak{X}_{K''}))\Longrightarrow H_{\dR}^{r+s}(X_{\Gamma}). 
 	\label{cover}
 	\end{align}
 \end{cor}
 \begin{proof}
 	Let $ \Gamma $ be a torsion free and discrete subgroup in $ G\times G' $, 
 	then the action of $ \Gamma $ on $ 
 	\mathfrak{X}_{K}\times_{K''}\mathfrak{X}_{K'}$ is discontinuous and free.	
 	Now the assertion follow from Theorem \ref{thm3.3} and Proposition 
 	\ref{prop3.5}.
 \end{proof}
 
 Let us compute the $ E_{2} $-terms in our covering spectral sequence 
 (\ref{cover}). First by corolary \ref{cor2.3}  we get
 \begin{align}
 H^{r}(\Gamma,H_{\dR}^{s}(\mathfrak{X}_{K}\times 
 _{K''}\mathfrak{X}_{K''}))&=H^{r}\left( 
 \Gamma,\bigoplus_{i+j=s}\text{Hom}_{K''}(V_{i}(K''),K'')\otimes_{K''}\text{Hom}_{K''}
 (V'_{j}(K''),K'') \right)  \notag \\
 &=\bigoplus_{i+j=s}H^{r}\left(\Gamma,\text{Hom}_{K''}(V_{i}(K''),K'')\otimes_{K''}\text{Hom}_{K''}
 (V'_{j}(K''),K'') \label{3.1}\right)
 \end{align}
 
 In order to compute (\ref{3.1}) we need some important theorems of Schneider 
 and 
 Stuhler. (\cite{1}, 
 Section 6). Let $ \mathcal{BT} $ (resp. $ \mathcal{BT'} $) be the Bruhat-Tits 
 building of 
 the group $ G $ (resp. $ G' $). (\cite{1}, Section 1) Let $ \sigma $ (resp. $ 
 \sigma' $) be a 
 simplex in $ \mathcal{BT} $ (resp. $ \mathcal{BT'} $). The stabilizer $ 
 B_{\sigma} $ of $ \sigma $ (resp. $ B_{\sigma'} $ of $ \sigma' $) has a 
 unique maximal normal pro-$ p $-subgroup $ U_{\sigma} $ (resp. $ U_{\sigma'} 
 $) 
 which itself is compact 
 open in $ G $ (resp. $ G' $).
 \begin{lem}(\cite{1}, Section 6)\label{lem3.7}
 	If $ A $ is any 
 	abelian group and $ I\subset \{1,\cdots,d\} $, then
 	\begin{enumerate}
 		\item  The $ G $-module $ V_{I}(A) 
 		$ 
 		has a projective resolution $ P^{\bullet} $ of 
 		finitely 
 		generated  $ 
 		A[G] $-modules
 		\[ 0\to \oplus_{\tau \in 
 			\mathcal{BT}_{d}}V_{I}(A)^{U_{\tau}}\to \cdots \to 
 		\oplus_{\sigma\in \mathcal{BT}_{0}}V_{I}(A)^{U_{\sigma}}\to V_{I}(A)\to 
 		0. 
 		\]
 		\item Furthermore, if $ \Gamma \subset G=PGL_{d+1}(K) $ is a cocompact 
 		discrete 
 		subgroup. Then the above resolution of $ V_{I}(A) $ is a projective 
 		resolution of finitely 
 		generated free $ 
 		A[\Gamma] $-modules.
 	\end{enumerate}	
 	
 \end{lem}
 
 The folloiwng propositon is the key fact to compute those Ext-groups at the $ 
 E_{2} $-terms in covering spectral sequence. (Compare \cite{1}, Proposition 4 
 in Section 5)
 \begin{prop}\label{prop3.8}
 	Let $ \Gamma$ be a discrete torsion-free cocompact 
 	subgroup of $ G\times G' $. Let $ I\subset \{1,\cdots,d\}, J\subset 
 	\{1,\cdots,d'\} $ and $ A $ be any commutative ring.
 	\begin{enumerate}
 		\item  $ V_{I}(A)\otimes_{A}V'_{J}
 		(A)$ has a projetive resolution by finitely generated free $ A[\Gamma] 
 		$-modules.
 		\item 	 We have
 		\[ H^{\ast}(\Gamma,\text{Hom}_{A}(V_{I}(A),A)\otimes_{A}\text{Hom}_{A}
 		(V'_{J}(A),A) 
 		)=H^{\ast}(\Gamma,\text{Hom}_{A}(V_{I}(A)\otimes_{A}V'_{J}(A),A)). \]
 		\item There is a natural exact 
 		sequence
 		\[ 0\to 
 		\text{Ext}^{r}_{\mathbb{Z}[\Gamma]}(V_{I}(\mathbb{Z})\otimes_{\mathbb{Z}}V'_{J}
 		(\mathbb{Z}),\mathbb{Z})
 		\otimes_{\mathbb{Z}}A\to 
 		\text{Ext}^{r}_{A[\Gamma]}(V_{I}(A)\otimes_{A}V'_{J}
 		(A),A)\to 
 		\text{Tor}_{\mathbb{Z}}(\text{Ext}^{r+1}_{\mathbb{Z}[\Gamma]}(V_{I}(\mathbb{Z})
 		\otimes_{\mathbb{Z}}V'_{J}(\mathbb{Z}) 
 		,A)\to 0 \]
 		for any $ r\geq 0 $.
 	\end{enumerate}
 \end{prop}
 \begin{proof}
 	By the Lemma \ref{lem3.7}, the $ G $-module $ V_{I}(A) $ (resp. $ G' 
 	$-module $ 
 	V'(J)(A) $) has a projective resolution of finitely generated $ A[G] 
 	$-modules 
 	(resp. $ 
 	A[G'] $-modules)
 	\[ P_{1}^{\bullet}\to V_{I}(A):~~ 0\to \oplus_{\tau \in 
 		\mathcal{BT}_{d}}V_{I}(A)^{U_{\tau}}\to \cdots \to 
 	\oplus_{\sigma\in \mathcal{BT}_{0}}V_{I}(A)^{U_{\sigma}}\to V_{I}(A)\to 0. 
 	\]
 	\[ P_{2}^{\bullet}\to V'_{J}(A):~~0\to \oplus_{\tau' \in 
 		\mathcal{BT'}_{d'}}V'_{J}(A)^{U_{\tau'}}\to \cdots \to 
 	\oplus_{\sigma'\in \mathcal{BT'}_{0}}V'_{J}(A)^{U_{\sigma'}}\to 
 	V'_{J}(A)\to 
 	0.  
 	\]
 	Since $ V_{I}(A) $ (resp. $ V'_{J}(A) $) is a free $ A $-module 
 	(\cite{1}, Corrollary 5 in Section 4) we get $ 
 	\text{Hom}_{A}(P_{1}^{\bullet},A) $ (resp. $ \text{Hom}_{A}(P_{2}
 	^{\bullet},A) $) is a resolution of $ \text{Hom}_{A}(V_{I}(A),A) $ 
 	(resp. $ \text{Hom}_{A}
 	(V'_{J}(A),A)  $). Then by the Kunneth formula for 
 	complexes (\cite{12}, Theorem 
 	3.6.3) we see that $ 
 	\text{Hom}_{A}(P_{1}^{\bullet},A)\otimes_{A}\text{Hom}_{A}(P_{2}
 	^{\bullet},A) $ is a resolution of $ 
 	\text{Hom}_{A}(V_{I}(A),A)\otimes_{A}\text{Hom}_{A}
 	(V'_{J}(A),A)  $. Similarly, $ 
 	\text{Hom}_{A}(P_{1}^{\bullet}\otimes_{A}P_{2}^{\bullet},A )$ is a 
 	resolution 
 	of $ 
 	\text{Hom}_{A}(V_{I}(A)\otimes_{A}V'_{J}(A),A) $ since $ 
 	V_{I}(A)\otimes_{A}V'_{J}(A) $ is also a free $ A $-module. 
 	
 	Since $ V_{I}(A) $ and $ V'_{J}(A) $ is a free $ A $-module, the $ A 
 	$-modules 
 	$ V_{I}(A)^{U_{\sigma}} $ and $ V'_{J}(A)^{U_{\sigma'}} $ are finitely 
 	generated and free. It follows that both $ 
 	P_{1}^{\bullet}$ and $ P_{2}^{\bullet}  $ are finitely generated 
 	free $ A $-modules. Then we have (\cite{13}, Corrollary 7 in Section 5, 
 	Chapter 
 	5)
 	\begin{align}
 	\text{Hom}_{A}(P_{1}^{\bullet},A)\otimes_{A}\text{Hom}_{A}(P_{2}
 	^{\bullet},A)\simeq 
 	\text{Hom}_{A}(P_{1}^{\bullet}\otimes_{A}P_{2}^{\bullet},A 
 	) .\label{3.2}
 	\end{align}
 	
 	The group $ G\times G' $ acts on the polysimplicial complexes $ 
 	\mathcal{BT}\times \mathcal{BT'} $. If 
 	$ (v,v') $ is a vertex of $ \mathcal{BT}\times \mathcal{BT'} $, then the 
 	stabilizer $ B_{(v,v')} $ of $ (v,v') $ is just $ B_{v}\times B_{v'} $.
 	Hence if $ \Gamma\subset 
 	G\times G' $ be a discrete torsion-free cocompact subgroup, then $ \Gamma $ 
 	acts freely on $ \mathcal{BT}\times \mathcal{BT'} $ so that $ 
 	P_{1}^{\bullet}\otimes_{A}P_{2}^{\bullet} $ are finitely generated free $ 
 	A[\Gamma] $-modules. It follows that $ 
 	\text{Hom}_{A}(P_{1}^{\bullet}\otimes_{A}P_{2}^{\bullet},A )$ are $ 
 	\Gamma 
 	$-injective resolutions. This implies the first assection. And the 
 	second 
 	assertion follows from (\ref{3.2}).
 	
 	For the third assertion, let $ P_{1}^{\bullet}\to V_{I}(\mathbb{Z}) $ 
 	(resp. $ 
 	P_{2}^{\bullet}\to 
 	V'_{J}(\mathbb{Z}) $) be a projective resolution by finitely generated $ 
 	\mathbb{Z}[G] $ (resp. $ 
 	\mathbb{Z}[G'] $)-modules. We have seen that $ 
 	P_{1}^{\bullet}\otimes_{\mathbb{Z}}P_{2}^{\bullet}\to 
 	V_{I}(\mathbb{Z})\otimes_{\mathbb{Z}}V'_{J}(\mathbb{Z})   $ is a projective 
 	resolution by finitely generated free $ \mathbb{Z}[\Gamma] $-modules. Then 
 	$ 
 	P_{1}^{\bullet}\otimes_{\mathbb{Z}}P_{2}^{\bullet}\otimes_{\mathbb{Z}}A\to 
 	V_{I}(A)\otimes_{\mathbb{Z}}V'_{J}(A) $ is a projective resolution by 
 	finitely 
 	generated free $ A[\Gamma] $-modules since $ 
 	V_{I}(\mathbb{Z})\otimes_{\mathbb{Z}}V'_{J}(\mathbb{Z}) $ is $ \mathbb{Z} 
 	$-free. 
 	
 	Therefore $ 
 	\text{Ext}^{\ast}_{\mathbb{Z}[\Gamma]}(V_{I}(\mathbb{Z})\otimes_{\mathbb{Z}}V'_{J}
 	(\mathbb{Z}),\mathbb{Z}) $, resp. $ 
 	\text{Ext}^{\ast}_{A[\Gamma]}(V_{I}(A)\otimes_{A}V'_{J}
 	(A),A) $, can be computed from the complex $ 
 	\text{Hom}_{\mathbb{Z}[\Gamma]}(P_{1}^{\bullet}\otimes_{\mathbb{Z}}P_{2}^{\bullet},\mathbb{Z})
 	$, resp. $ 
 	\text{Hom}_{A[\Gamma]}(P_{1}^{\bullet}\otimes_{\mathbb{Z}}P_{2}^{\bullet}\otimes_{\mathbb{Z}}A,A)
 	$. The third assertion follows from the universal 
 	coefficient theorem applied to the complex
 	\[ 
 	\text{Hom}_{A[\Gamma]}(P_{1}^{\bullet}\otimes_{\mathbb{Z}}P_{2}^{\bullet}\otimes_{\mathbb{Z}}A,A)
 	=\text{Hom}_{\mathbb{Z}[\Gamma]}(P_{1}^{\bullet}\otimes_{\mathbb{Z}}P_{2}^{\bullet},\mathbb{Z})
 	\otimes_{\mathbb{Z}}A . \]
 \end{proof}
 
 \begin{lem}(\cite{6},Lemma 5.1 and Lemma 5.2)\label{lem3.9}
 	Let $ R $ be a commutative ring and $ \Gamma $ be any group. Let $ V,M $  
 	be 
 	$ R[\Gamma] 
 	$-modules, and assume that $ V $ is free as an $ R $-module. Then
 	\begin{enumerate}
 		\item We have
 		\[ H^{\bullet}(\Gamma,\text{Hom}_{R}(V,M))\simeq 
 		\text{Ext}^{\bullet}_{R[\Gamma]}(V,M). \]
 		\item 	Moreover assume that $ V $ has a resolution $ F^{\bullet}\to V 
 		$ where for 
 		every $ n\geq 0,~F_{n} $ is a finitely generated free $ R[\Gamma] 
 		$-module. 
 		Let $ S $ be a flat $ R $-algebra. Then the canonical map
 		\[ \text{Ext}_{R[\Gamma]}^{\bullet}(V,M)\otimes_{R}S\to 
 		\text{Ext}_{S[\Gamma]}^{\bullet}(V\otimes_{R}S,M\otimes_{R}S) \]
 		is an isomorphism.
 	\end{enumerate}
 \end{lem}
 
 Going back to the $ E_{2} $-terms (\ref{3.1}) of our covering spectral 
 sequence .
 Because of
 Proposition 
 \ref{prop3.8} and Lemma \ref{lem3.9} we can rewrite it in the form 
 \begin{align*}
 H^{r}(\Gamma,H_{\dR}^{s}(\mathfrak{X}_{K}\times 
 _{K''}\mathfrak{X}_{K''}))&=\bigoplus_{i+j=s}H^{r}(\Gamma,\text{Hom}_{K''}(V_{i}(K''),K'')\otimes_{K''}\text{Hom}_{K''}
 (V'_{j}(K''),K'') )\\
 &=\bigoplus_{i+j=s}H^{r}(\Gamma,\text{Hom}_{K''}(V_{i}(K'')\otimes_{K''}
 V'_{j}(K''),K'')\\
 &=\bigoplus_{i+j=s}\text{Ext}_{K''[\Gamma]}^{r}(V_{i}(K'')\otimes_{K''}
 V'_{j}(K''),K'')
 \end{align*}
 so that for $ 0\leq s\leq d+d' $ the covering spectral sequence becomes 
 \begin{align}
 E_{2}^{r,s}=\bigoplus_{i+j=s}\text{Ext}_{K''[\Gamma]}^{r}\left(V_{i}(K'')\otimes_{K''}
 V'_{j}(K''),K''\right)\Longrightarrow H^{r+s}(X_{\Gamma}). \label{coverss}
 \end{align}
 
 We want to study the groups 
 $\text{Ext}_{K''[\Gamma]}^{r}(V_{i}(K'')\otimes_{K''}
 V'_{j}(K''),K'')  $ in (\ref{coverss}).
 
 Because of the Proposition \ref{prop3.8} this amounts to 
 the computation of
 \[ \text{Ext}_{\mathbb{C}[\Gamma]}^{\ast}(V_{i}(\mathbb{C})\otimes_{\mathbb{C}}
 V'_{j}(\mathbb{C}),\mathbb{C}). \]
 
 Since $ \Gamma $ is cocompact the $ G\times G' $-representation
 \[ \text{Ind}_{\Gamma}:=C^{\infty}(G\times G'/\Gamma,\mathbb{C}) \]
 is admissible. By Shapiro's lemma we have
 \[ \text{Ext}_{\mathbb{C}[\Gamma]}^{\ast}(V_{i}(\mathbb{C})\otimes_{\mathbb{C}}
 V'_{j}(\mathbb{C}),\mathbb{C})=\text{Ext}^{\ast}_{G\times 
 	G'}(V_{i}(\mathbb{C})\otimes_{\mathbb{C}}
 V'_{j}(\mathbb{C}), \text{Ind}_{\Gamma}) \]
 The representation $ \text{Ind}_{\Gamma} $ is a unitary representations, and 
 decomposes into the direct sum of irreducible smooth unitary 
 representations of $ G\times G' $ with finite multiplicities (see \cite{1}, 
 Section 5).
 
 \begin{prop}(\cite{9}, Lemma 2.1)\label{prop3.10}
 	\begin{enumerate}
 		\item For $ I_{1},I_{2}\subset \{1,\cdots,d\} $ and $ 
 		J_{1},J_{2}\subset 
 		\{1,\cdots,d'\} $, we have
 		\begin{align*}
 		\text{Ext}_{G\times G'}^{i}(V_{I_{1}}(\mathbb{C})\otimes 
 		V'_{J_{1}}(\mathbb{C}),V_{I_{2}}(\mathbb{C})\otimes 
 		V'_{J_{2}}(\mathbb{C} )=
 		\begin{cases}
 		\mathbb{C} & if~i=\Delta(I_{1},I_{2})+\Delta(J_{1},J_{2})\\
 		0&otherwise
 		\end{cases}
 		\end{align*}
 		where $ \Delta(I_{1},I_{2})=|(I_{1}\cup I_{2})|-|(I_{1}\cap I_{2})| $.
 		
 		\item Let $ V\otimes V' $ be an irreducible smooth representation of $ 
 		G\times G' $. If it is not of the form $ V_{I_{0}}(\mathbb{C})\otimes 
 		V'_{J_{0}}(\mathbb{C}) $  with $ I_{0}\subset \{1,\cdots,d\} $ 
 		and $ J_{0}\subset \{1,\cdots,d'\} $, then we have 
 		\[ \text{Ext}_{G\times 
 			G'}^{i}(V_{I}(\mathbb{C})\otimes 
 		V'_{J}(\mathbb{C}),V\otimes V' )=0  \]
 		for every $ I\subset 
 		\{1,\cdots,d\},J\subset \{1,\cdots,d'\} $ and $ i  $.
 	\end{enumerate}
 \end{prop}
 
 \begin{thm}\label{thm3.11}
 	Let $ \Gamma $ be a discrete torsion-free cocompact subgroup of $ G\times 
 	G' $. 
 	Let $ m_{1,0} $ (resp. $ m_{0,1} $, resp. $ m_{1,1} $) be the multiplicity 
 	of the representation $ 
 	V_{\emptyset}\otimes \mathbb{C} $ (resp. $ \mathbb{C}\otimes 
 	V'_{\emptyset}  $, resp. $ V_{\emptyset}\otimes V'_{\emptyset}$) in the 
 	representation $ \text{Ind}_{\Gamma} $ of $ G\times G' $. 
 	If 
 	$ d,d' $ are all even numbers and $ d\leq d' $, we have
 	\begin{align*}
 	E_{2}^{r,s}&=\bigoplus_{i+j=s}\text{Ext}_{K''[\Gamma]}^{r}\left(V_{i}(K'')\otimes_{K''}
 	V'_{j}(K''),K''\right)= \begin{cases}
 	K'' & r=s<\frac{d}{2}\\
 	K''^{m_{1,0}+1} & r=s,~\frac{d}{2}\leq r<\frac{d'}{2}\\
 	K''^{m_{1,0}+m_{0,1}+1} & r=s\neq \frac{d+d'}{2},r\geq \frac{d'}{2}\\
 	K''^{m_{1,1}+m_{1,0}+m_{0,1}}  & r+s=d+d',r\neq \frac{d+d'}{2}\\
 	K''^{m_{1,1}+m_{1,0}+m_{0,1}+1}& r=s= \frac{d+d'}{2}\\
 	K''^{m_{1,0}}& r+s\in 2\mathbb{Z},d\leq r+s<d',r\neq s,r+s\neq d+d'\\
 	K''^{m_{1,0}+m_{0,1}}& r+s\in 2\mathbb{Z},r+s\geq d',r\neq s,r+s\neq d+d'\\
 	0& otherwise
 	\end{cases}
 	\end{align*}
 \end{thm}
 \begin{proof}
 	The proof is the same as \cite{1}, Theorem 3 in Section 5.
 	
 	Assume that $ V_{I}(\mathbb{C})\otimes 
 	V'_{J}(\mathbb{C}) $ appears in $ \text{Ind}_{\Gamma} $, then it is unitary 
 	and 
 	thus $  V_{I}(\mathbb{C}) $ and $ V'_{J}(\mathbb{C}) $ are unitary. Hence 
 	we 
 	conclude that $ I $(resp. $ J $) is either $ \emptyset $ or $ 
 	\{1,\cdots,d\} $ 
 	(resp. $ \{1,\cdots,d'\} $). Note that the multiplicity of the trivial 
 	representation $ \mathbb{C}\otimes \mathbb{C} $ in $ \text{Ind}_{\Gamma} $ 
 	equals $ 1 $. As 
 	in 
 	the proof in (\cite{1}, Theorem 3), by Proposition \ref{prop3.10} we have
 	\begin{align*}
 	&\text{Ext}_{\mathbb{C}[\Gamma]}^{\ast}\left(V_{i}(\mathbb{C})\otimes
 	V'_{j}(\mathbb{C}),\mathbb{C}\right)\\
 	=&\text{Ext}^{\ast}_{G\times 
 		G'}\left(V_{i}(\mathbb{C})\otimes
 	V'_{j}(\mathbb{C}), \text{Ind}_{\Gamma}\right)\\
 	=&\text{Ext}^{\ast}_{G\times 
 		G'}\left( V_{i}(\mathbb{C})\otimes
 	V'_{j}(\mathbb{C}),(\mathbb{C}\otimes \mathbb{C})\oplus 
 	(V_{\emptyset}\otimes \mathbb{C})^{m_{1,0}}\oplus 
 	(\mathbb{C}\otimes V'_{\emptyset})^{m_{0,1}}\oplus (V_{\emptyset}\otimes 
 	V'_{\emptyset})^{m_{1,1}}\right) 
 	\\
 	=&\text{Ext}^{\ast}_{G\times 
 		G'}\left( V_{i}(\mathbb{C})\otimes
 	V'_{j}(\mathbb{C}), \mathbb{C}\otimes \mathbb{C}\right) \oplus  
 	\text{Ext}^{\ast}_{G\times 
 		G'}\left(V_{i}(\mathbb{C})\otimes
 	V'_{j}(\mathbb{C}), V_{\emptyset}(\mathbb{C})\otimes 
 	\mathbb{C}\right)^{m_{1,0}}\\
 	&\qquad \oplus 
 	\text{Ext}^{\ast}_{G\times 
 		G'}\left(V_{i}(\mathbb{C})\otimes
 	V'_{j}(\mathbb{C}),\mathbb{C}\otimes 
 	V'_{\emptyset}(\mathbb{C})\right)^{m_{0,1}} 
 	\oplus 
 	\text{Ext}^{\ast}_{G\times 
 		G'}\left(V_{i}(\mathbb{C})\otimes
 	V'_{j}(\mathbb{C}),V_{\emptyset}(\mathbb{C})\otimes 
 	V'_{\emptyset}(\mathbb{C})\right)^{m_{1,1}}.
 	\end{align*}
 	and
 	\begin{align*}
 	\text{Ext}^{r}_{G\times 
 		G'}\left(V_{i}(\mathbb{C})\otimes
 	V'_{j}(\mathbb{C}), \mathbb{C}\otimes \mathbb{C}\right)=\begin{cases}
 	\mathbb{C} &r=i+j\\
 	0&otherwise
 	\end{cases}
 	\end{align*}
 	
 	\begin{align*}
 	\text{Ext}^{r}_{G\times 
 		G'}\left(V_{i}(\mathbb{C})\otimes
 	V'_{j}(\mathbb{C}), V_{\emptyset}(\mathbb{C})\otimes 
 	\mathbb{C}\right)=\begin{cases}
 	\mathbb{C} &r=d-i+j\\
 	0&otherwise
 	\end{cases}
 	\end{align*}
 	
 	\begin{align*}
 	\text{Ext}^{r}_{G\times 
 		G'}\left(V_{i}(\mathbb{C})\otimes
 	V'_{j}(\mathbb{C}), \mathbb{C}\otimes 
 	V'_{\emptyset}(\mathbb{C})\right)=\begin{cases}
 	\mathbb{C} &r=i+d'-j\\
 	0&otherwise
 	\end{cases}
 	\end{align*}
 	
 	\begin{align*}
 	\text{Ext}^{r}_{G\times 
 		G'}\left(V_{i}(\mathbb{C})\otimes
 	V'_{j}(\mathbb{C}), V_{\emptyset}(\mathbb{C})\otimes 
 	V'_{\emptyset}(\mathbb{C})\right)=\begin{cases}
 	\mathbb{C} &r=d+d'-i-j\\
 	0&otherwise
 	\end{cases}
 	\end{align*}
 	
 	As a result, if $ d,d' $ are all even numbers we see that in our 
 	spectral sequence 
 	nonvanishing $ 
 	E_{2} $-terms only occur on the lines $ r=s $, $ r+s=d+d' $ and those 
 	pairs 
 	$ (r,s) $ satisfying  $  
 	r+s\in 
 	2\mathbb{Z},r+s\geq d $ or $ r+s\in 2\mathbb{Z},r+s\geq d' $. Altogether 
 	we obtain the assertion. 
 	
 \end{proof}
 
 \begin{cor}
 	If 
 	$ d,d' $ are all even numbers and $ d\leq d' $, the covering spectral 
 	sequence 
 	for $ 
 	H_{dR}^{n}(X_{\Gamma}) $ degenerates at 
 	$ E_{2} $. Thus we have
 	\begin{align*}
 	H_{\dR}^{n}(X_{\Gamma})=\begin{cases}
 	K''& 0\leq n< d,~n~\text{even}\\
 	K''^{(n+1)m_{1,0}+1}& d\leq n< d',~n~\text{even}\\
 		K''^{(n+1)(m_{1,0}+m_{0,1})+1}& d'\leq n<d+d',~n~\text{even}\\
 	K''^{(d+d'+1)(m_{1,1}+m_{1,0}+m_{0,1})+1} &n=d+d'\\
 	K''^{(2d+2d'+1-n)(m_{1,0}+m_{0,1})+1} &d+d'<n\leq 2(d+d'),~n~\text{even}\\
 	0&otherwise
 	\end{cases}
 	\end{align*}
 \end{cor}
 \begin{proof}
 	In the proof of Theorem \ref{thm3.11} we have seen that in our 
 	spectral sequence 
 	nonvanishing $ 
 	E_{2}$-terms only occur on the lines $ r=s $, $ r+s=d+d' $ and those 
 	pairs 
 	$ (r,s) $ satisfying  $  
 	r+s\in 
 	2\mathbb{Z},r+s\geq d $ or $ r+s\in 2\mathbb{Z},r+s\geq d' $. Then all 
 	differentials in the spectral sequence must be zero. Thus our 
 	covering spectral seuqence degenerates at $ E_{2} $.
 \end{proof}
 
 \begin{cor}\label{cor3.13}
 	If the above covering spectral sequence degenerates at $ E_{2} $, then
     we have
 	\begin{align}
 	\text{dim}(F_{\Gamma}^{i})+\text{dim}(F_{\Gamma}^{n+1-i})=\text{dim}(H_{\dR}
 	^{n}(X_{\Gamma}))
 	\end{align}
 	for any $ 0\leq i\leq n+1 $.
 \end{cor}
 
 We will prove 
 our covering spectral sequence always degenerates at $ E_{2} $ even if $ d,d' 
 $ 
 are 
 not even in next section.
 
 \section{ Hodge-type decomposition for $ H^{n}_{\dR}(X_{\Gamma}) $}
 Let $ \Gamma$ be a discrete torsion-free cocompact 
 subgroup of $ G\times G' $, we have seen that the quotient 
 $X_{\Gamma}:=\Gamma \backslash 
 (\mathfrak{X}_{K}\times_{K''}\mathfrak{X}_{K'})$ becomes a smooth 
 rigid 
 variety 
 over $ K'' $ and we have also proved the existence of covering 
 spectral sequence for $ H_{\dR}^{n}(X_{\Gamma}) $. The associated 
 filtration is denoted by $ F_{\Gamma}^{\bullet} $. 
 
 There is a second natural filtration on $ H^{n}_{\dR}(X_{\Gamma}) $, the 
 Hodge 
 filtration $ F_{\dR}^{\bullet} $. It is induced by the Hodge-de Rham 
 spectral 
 sequence
 \[ E_{1}^{r,s}=H^{s}(X_{\Gamma},\Omega^{r}_{X_{\Gamma}/K''})\Longrightarrow 
 H_{\dR}^{n}(X_{\Gamma}). \]
 By (\cite{3}, Appendix), the 
 analytic 
 variety $ X_{\Gamma} $ is algebraizable to a projective variety over $ K'' 
 $. The GAGA-principle then implies the de Rham cohomology of $ X_{\Gamma} $ 
 are equal to the corresponding algebraic cohomology groups. Since $ K $ 
 has characteristic $ 0 $ we see that the Hodge de Rham spectral sequence 
 for $ 
 X_{\Gamma} $ degenerates. 
 
 One says that the filtrations $ F_{\dR}^{\bullet} $ and $ 
 F_{\Gamma}^{\bullet} $ are \textit{opposite} if for every $ 0\leq p\leq n+1 
 $,
 \[ H^{n}=F_{\dR}^{p}H^{n}\oplus F_{\Gamma}^{n+1-p}H^{n}. \]
 Setting, for $ p+q=n $,
 \[ H^{p,q}:=F^{p}_{\dR}H^{n}\cap F_{\Gamma}^{q}H^{n}. \]
 If the two filtrations are opposite, then we have a Hodge-type 
 decomposition for $ X_{\Gamma} $
 \[ H^{n}_{\dR}(X_{\Gamma})=\bigoplus_{p+q=n}H^{p,q}. \]
 
 Let $ (K^{r,s},d',d'') $ be a double complexes of $ K $-vector spaces 
 vanishing for $ r<0 $ or $ s<0 $ and let $ 'F^{\bullet}H^{n} 
 $, resp. $ ''F^{\bullet}H^{n} $, denote the first filtration induced by 
 the first spectral sequence, resp. the 
 second filtration induced by the second spectral sequence, of $ H^{n} $.
 Let us give a general criterion for the two filtrations $ 
 'F^{\bullet}H^{n},''F^{\bullet}H^{n} $ to be opposite. (See \cite{5}, Section 
 3)
 
 \begin{lem}\label{lem4.1}
 	If  $ 
 	\varphi:\widetilde{K}^{r,s}\to K^{r,s} $ be 
 	a map of double 
 	complexes and assume that for $  \widetilde{K}^{r,s} $, we have $ 
 	\tilde{d}'=0 
 	$ throughout. Assume that the maps $ 
 	''E_{2}\varphi:''\tilde{E}_{2}^{r,s}\to 
 	''E_{2}^{r,s} $ are isomorphisms. 
 	\begin{enumerate}
 		\item (\cite{5}, Lemma 3.2)\label{lem4.2} 	Let $  \tilde{H}^{n},H^{n} 
 		$ be 
 		the 
 		homologies of the total complexes. Then the induced map $ 
 		H\varphi:\tilde{H}^{n}\to H^{n} $ is an isomorphism for every $ n 
 		$, 
 		under with $ ''\tilde{F}^{p}\tilde{H}^{n} $ and $  ''F^{p}H^{n} $ 
 		agree, 
 		and the spectral sequence $ ''E_{2}^{r,s}\Longrightarrow H^{n} $ 
 		degenerates at $ ''E_{2} $. 
 		\item We have $ 
 		H^{n}='F^{i}+''F^{n+1-i} $ for any $ 
 		0\leq i\leq n+1 $.
 		\item Furthermore, if for any $ 0\leq i\leq n+1 $ one has $ 
 		\text{dim}('F^{i})+\text{dim}('F^{n+1-i})=\text{dim}(H^{n}) $ and $ 
 		\text{dim}(''F^{i})+\text{dim}(''F^{n+1-i})=\text{dim}(H^{n}) $. 
 		Then the filtrations $ 'F^{\bullet} $ and $ ''F^{\bullet} $ on $ 
 		H^{n} $ are opposite to each other.
 	\end{enumerate}
 \end{lem}
 \begin{proof}
  It suffices to show the third assertion. Since
 	\[ \text{dim}('F^{i})\geq 
 	\text{dim}(H^{n})-\text{dim}(''F^{n+1-i})=\text{dim}(''F^{i}) 
 	\]
 	and 
 	\[ \text{dim}(''F^{i})\geq 
 	\text{dim}(H^{n})-\text{dim}('F^{n+1-i})=\text{dim}('F^{i}) 
 	\]
 	we get $ 
 	\text{dim}('F^{i})=\text{dim}(''F^{i}) $ for any $ 0\leq i\leq 
 	n+1 $. Therefore we have $  H^{n}='F^{i}\oplus ''F^{n+1-i} $.
 \end{proof}
 
 We want to prove that the
 covering filtration $ F_{\Gamma}^{\bullet} $ is 
 opposite to the Hodge filtration $ F_{\dR}^{\bullet} $ 
 on $ 
 H^{n}_{\dR}(X_{\Gamma}) $. 
 
 First we recall the results of Iovita and Spiess \cite{6}. We denote the space 
 of \textit{bounded logarithmic} differential $ p $-forms 
 on $ 
 \mathfrak{X}_{K} $ by $ 
 \Omega_{log,b}^{p}(\mathfrak{X}_{K})$ and such forms are closed forms.  
 The main 
 theorem in \cite{6} 
 states 
 that the composition map
 \[ \Omega_{log,b}^{p}(\mathfrak{X}_{K}) \hookrightarrow  
 \Omega^{p}(\mathfrak{X}_{K})\to H^{p}(\mathfrak{X}_{K}),\qquad w\mapsto [w] 
 \]    
 is injective and its image corresponds to the inclusion
 \[ 
 \text{Hom}_{\mathbb{Z}}(V_{p}(\mathbb{Z}),\mathcal{O}_{K})\otimes_{
 	\mathcal{O}_{K}}K\to
 \text{Hom}_{\mathbb{Z}}(V_{p}(\mathbb{Z}),K) \]
 
 The following proposition is the key ingredient in our proof for Hodge-type 
 decomposition (compare Proposition 5.3 in \cite{6}).
 
 \begin{prop}\label{prop4.1}
 	Let $ p_{1},p_{2} $ be positive integers. Then the canonical map 
 	\[ 
 	(\Omega_{log,b}^{p_{1}}(\mathfrak{X}_{K})\otimes_{K}K'')\otimes_{K''}(\Omega_{log,b}^{
 		p_{2}}(
 	\mathfrak{X}_{K'})\otimes_{K'}K'')
 	\to 
 	H^{p_{1}}_{\dR}(\mathfrak{X}_{K}\otimes_{K}K'')\otimes_{K''} 
 	H^{p_{2}}_{\dR}(\mathfrak{X}_{K'}\otimes_{K'}K'')  \]
 	induce an isomorphism on cohomology groups
 	\[ H^{\ast}(\Gamma,  
 	(\Omega_{log,b}^{p_{1}}(\mathfrak{X}_{K})\otimes_{K}K'')\otimes_{K''}(\Omega_{log,b}^{
 		p_{2}}(
 	\mathfrak{X}_{K'})\otimes_{K'}K''))\to H^{\ast}(\Gamma,  
 	H^{p_{1}}_{\dR}(\mathfrak{X}_{K}\otimes_{K}K'')\otimes_{K''} 
 	H^{p_{2}}_{\dR}(\mathfrak{X}_{K'}\otimes_{K'}K'') ). \]
 \end{prop}
 \begin{proof}
 	By Proposition \ref{prop3.8} and Lemma \ref{lem3.9} we have
 	\begin{align*}
 	&H^{\ast}(\Gamma,  
 	(\Omega_{log,b}^{p_{1}}(\mathfrak{X}_{K})\otimes_{K}K'')\otimes_{K''}(\Omega_{log,b}^{
 		p_{2}}(
 	\mathfrak{X}_{K'})\otimes_{K'}K''))\\
 	=&Ext_{K''[\Gamma]}^{\ast} \left( K'',
 	(Hom_{\mathcal{O}_{K}}(V_{p_{1}}(\mathcal{O}_{K}),\mathcal{O}_{K})
 	\otimes_{\mathcal{O}_{K}}K'') \otimes_{K''}
 	( Hom_{\mathcal{O}_{K'}}(V'_{p_{2}}(\mathcal{O}_{K'}),\mathcal{O}_{K'} 
 	)\otimes_{\mathcal{O}_{K'}}K'') \right) \\
 	=&Ext_{K''[\Gamma]}^{\ast} \left( K'',
 	(Hom_{\mathcal{O}_{K''}}(V_{p_{1}}(\mathcal{O}_{K''}),\mathcal{O}_{K''})
 	\otimes_{\mathcal{O}_{K''}}K'') \otimes_{K''}
 	( 
 	Hom_{\mathcal{O}_{K'}}(V'_{p_{2}}(\mathcal{O}_{K''}),\mathcal{O}_{K''} 
 	)\otimes_{\mathcal{O}_{K''}}K'') \right) \\
 	=&Ext_{K''[\Gamma]}^{\ast}(K'',
 	Hom_{\mathcal{O}_{K''}}(V_{p_{1}}(\mathcal{O}_{K''})\otimes 
 	V'_{p_{2}}(\mathcal{O}_{K''}),\mathcal{O}_{K''})\otimes_{\mathcal{O}_{K''}}K'')\\
 	=&Ext_{\mathcal{O}_{K''}[\Gamma]}^{\ast}(\mathcal{O}_{K''},
 	Hom_{\mathcal{O}_{K''}}(V_{p_{1}}(\mathcal{O}_{K''})\otimes_{\mathcal{O}_{K''}}
 	V'_{p_{2}}(\mathcal{O}_{K''}),\mathcal{O}_{K''}))\otimes_{\mathcal{O}_{K''}}K''\\
 	=&Ext_{\mathcal{O}_{K''}[\Gamma]}^{\ast}(V_{p_{1}}(\mathcal{O}_{K''})
 	\otimes_{\mathcal{O}_{K''}}
 	V'_{p_{2}}(\mathcal{O}_{K''}),\mathcal{O}_{K''})\otimes_{\mathcal{O}_{K''}}K''\\
 	=&Ext_{K''[\Gamma]}^{\ast}(V_{p_{1}}(K'')\otimes_{K''}
 	V'_{p_{2}}(K''),\mathcal{O}_{K''}\otimes_{\mathcal{O}_{K''}}K'')\\
 	=& H^{\ast}(\Gamma,  
 	H^{p_{1}}_{\dR}(\mathfrak{X}_{K}\otimes_{K}K'')\otimes_{K''} 
 	H^{p_{2}}_{\dR}(\mathfrak{X}_{K'}\otimes_{K'}K'') )
 	\end{align*}
 	since $ \mathcal{O}_{K''} $ is a free $ \mathcal{O}_{K} $-module (resp. $ 
 	\mathcal{O}_{K'} $-module) of finite rank.
 	This completes the proof.
 \end{proof}
 
 Now, we shall prove the main theorem (Theorem \ref{thm1.1}) of this paper.
 
 \textbf{Proof of the Theorem \ref{thm1.1}.}
 Since $ X:=\mathfrak{X}_{K}\times_{K''}\mathfrak{X}_{K'} $ is a 
 Stein-space, 
 its de Rham cohomology can be computed from the complex $ 
 \Omega^{\bullet}(X) $ of global holomorphic differential forms on $ X $. 
 Hence we have
 \[ 
 H^{\ast}(\Gamma,H_{\dR}^{s}(X))=H^{\ast}(\Gamma,h^{s}(\Omega^{\bullet}(X))),
 \]
 \[ 
 H^{\ast}(X_{\Gamma},\Omega^{r}_{X_{\Gamma}/K''})=H^{\ast}(\Gamma,\Omega^{r}(X)),
 \]
 \[ H_{\dR}^{\ast}(X_{\Gamma})=H^{\ast}(\Gamma,\Omega^{\bullet}(X)). \]
 Consider the hypercohomology of complex $ \Omega^{\bullet}(X) $ with 
 respect to the functor $ F:=(-)^{\Gamma} $ from the category of $ \Gamma 
 $-modules to the category of $ K'' $-vector spaces. As usual, first we 
 choose a 
 Cartan-Eilenberg resolution $ I^{\bullet,\bullet} $ of complex 
 $ \Omega^{\bullet}(X) $. This is a double complex $ (I^{r,s},d',d'') $ of 
 injective $ \Gamma $-modules and we denote the double complex $ 
 F(I^{r,s})=(I^{r,s})^{\Gamma} $ by $ (K^{r,s},d',d'') $. Now the homology 
 of total 
 complex $ 
 \text{Tot}(K^{r,s}) $ is just $ H_{\dR}^{n}(X_{\Gamma}) $. Then there are two 
 spectral sequences abutting to $ H_{\dR}^{n}(X_{\Gamma}) $.
 
 The first is 
 \[ 
 'E_{1}^{p,q}=R^{q}F(\Omega^{p}(X))=H^{q}(\Gamma,\Omega^{p}(X))=H^{q}(X_{\Gamma},
 \Omega^{p}(X)) \]
 and the associated filtration is Hodge filtration $ 
 F_{\dR}^{\bullet} $. The second is 
 \[ 
 ''E_{2}^{p,q}=R^{p}F(h^{q}(\Omega^{\bullet}(X)))=H^{p}(\Gamma,H_{\dR}^{q}(X))
 \]
 and the associated filtration is covering filtration $ F_{\Gamma}^{\bullet} 
 $.
 
 Similarly, consider the hypercohomology of complex 
 \[ 
 (\Omega_{log,b}^{\bullet}(\mathfrak{X}_{K})\otimes_{K}K'')\otimes_{K''}(\Omega_{log,b}^{
 	\bullet}(
 \mathfrak{X}_{K'})\otimes_{K'}K'')  \]
 with repsect to the functor $ F $ and 
 the 
 corresponding 
 double complex is denoted by $ 
 ( \tilde{K}^{r,s},\tilde{d}',\tilde{d}'') $. Since all differentials in $ 
 (\Omega_{log,b}^{\bullet}(\mathfrak{X}_{K})\otimes_{K}K'')\otimes_{K''}(\Omega_{log,b}^{
 	\bullet}(
 \mathfrak{X}_{K'})\otimes_{K'}K'') $ are zero, we see that $ d'=0 $.
 There are also two spectral sequence abutting 
 to $ H^{n}(\Gamma,  
 (\Omega_{log,b}^{\bullet}(\mathfrak{X}_{K})\otimes_{K}K'')\otimes_{K''}(\Omega_{log,b}^{
 	\bullet}(
 \mathfrak{X}_{K'})\otimes_{K'}K'')) $. The second is 
 \[ ''\tilde{E}_{2}^{p,q}=H^{p}\left( \Gamma, \oplus_{p_{1}+p_{2}=q} 
 (\Omega_{log,b}^{p_{1}}(\mathfrak{X}_{K})\otimes_{K}K'')\otimes_{K''}(\Omega_{log,b}^{
 	p_{2}}(
 \mathfrak{X}_{K'})\otimes_{K'}K'')\right)  \]

 The inclusion of complexes
 \[  
 (\Omega_{log,b}^{\bullet}(\mathfrak{X}_{K})\otimes_{K}K'')\otimes_{K''}(\Omega_{log,b}^{
 	\bullet}(
 \mathfrak{X}_{K'})\otimes_{K'}K'')\hookrightarrow 
 \Omega^{\bullet}(\mathfrak{X}_{K}\times_{K''}\mathfrak{X}_{K'})=
 \Omega^{\bullet}(\mathfrak{X}_{K}\otimes_{K}K'')\otimes_{K''}\Omega^{\bullet}(
 \mathfrak{X}_{K'}\otimes_{K'}K'')  \]
 induces a map of double complexes $ 
 \varphi:\tilde{K}^{r,s}\to K^{r,s} $. By Proposition \ref{prop2.2} and 
 Proposition \ref{prop4.1} we see that the maps 
 \[''E_{2}\varphi:''\tilde{E}_{2}^{p,q}=H^{p}\left( \Gamma, 
 \oplus_{p_{1}+p_{2}=q} 
 (\Omega_{log,b}^{p_{1}}(\mathfrak{X}_{K})\otimes_{K}K'')\otimes_{K''}(\Omega_{log,b}^{
 	p_{2}}(
 \mathfrak{X}_{K'})\otimes_{K'}K'')\right)\to
 H^{p}(\Gamma,H_{\dR}^{q}(X))=''E_{2}^{p,q}
 \]
 are isomorphisms.
 By Lemma 
 \ref{lem4.1} our covering spectral sequence for $ 
 H_{\dR}^{n}(X_{\Gamma}) $ degenerates at 
 $ E_{2} $ and we have
 \begin{align}
 H_{\dR}^{n}(X_{\Gamma})=F_{\Gamma}^{i}+ F_{\dR}^{n+1-i} \label{4}
 \end{align}
 for any $ 
 0\leq i\leq n+1 $. 
 
 Finally, by Corollary \ref{cor3.13} we get
 \begin{align}
 \text{dim}(F_{\Gamma}^{i})+\text{dim}(F_{\Gamma}^{n+1-i})=\text{dim}(H_{\dR}
 ^{n}(X_{\Gamma})). \label{5}
 \end{align}
 and by Serre duality (\cite{10}, Corollary 7.13) we get
 \begin{align}
 \text{dim}(F_{\dR}^{i})+\text{dim}(F_{\dR}^{n+1-i})=\text{dim}(H_{\dR}
 ^{n}(X_{\Gamma})). \label{6}
 \end{align}
 Now the assection follows from (\ref{4}), (\ref{5}), (\ref{6}) and Lemma 
 \ref{lem4.1}. $\qquad \qquad \qquad \qquad \qquad \qquad  \square $
 
  \section*{Acknowledgement}
   The content of this paper is a part of my master's thesis at East China 
   Normal University. I would like to thank my advisor Bingyong Xie who 
   introduced me to the subjects treated here and who offered useful 
   suggestions. Without his
   encouragement, this paper would never have been written.

 \medskip
 
 \bibliography{sample}
\end{document}